\documentclass[a4paper,12pt, reqno]{amsart}
\usepackage{latexsym}
\usepackage{color}
\usepackage{amssymb,amsfonts,amsmath,mathrsfs}
\newtheorem{theorem}{Theorem}
\newtheorem{lemma}[theorem]{Lemma}
\newtheorem{corollary}[theorem]{Corollary}

\newtheorem{proposition}[theorem]{Proposition}

\newcommand{\kommentar}[1]{}
\renewcommand{\pmod}[1]{\,(\mathrm{mod}\,#1)}

\numberwithin{theorem}{section} \numberwithin{equation}{section}

\allowdisplaybreaks

\begin{document}

\title[Dirichlet $L$--functions with prime conductors]{Moments of Dirichlet $L$--functions with prime conductors over function fields}

\author{Hung M. Bui and Alexandra Florea}
\address{Department of Mathematics, University of Manchester, Manchester M13 9PL, UK}
\email{hung.bui@manchester.ac.uk}
\address{Department of Mathematics, Columbia University, New York NY 10027, USA}
\email{aflorea@math.columbia.edu}

\maketitle
\begin{abstract}
We compute the second moment in the family of quadratic Dirichlet $L$--functions with prime conductors over $\mathbb{F}_q[x]$ when the degree of the discriminant goes to infinity, obtaining one of the lower order terms. We also obtain an asymptotic formula with the leading order term for the mean value of the derivatives of $L$--functions associated to quadratic twists of a fixed elliptic curve over $\mathbb{F}_q(t)$ by monic irreducible polynomials, which allows us to show that there exists a monic irreducible polynomial such that the analytic rank of the corresponding twisted elliptic curve is equal to $1$.
\end{abstract}
\section{Introduction}
In this paper we study the family of $L$--functions $L(s,\chi_P)$ as $P$ ranges over monic, irreducible polynomials of degree $2g+1$ over $\mathbb{F}_q[x]$ and the family $L(E \otimes \chi_P,s)$ for $E/\mathbb{F}_q(t)$ a fixed elliptic curve, again as $P$ ranges over monic, irreducible polynomials.

Andrade and Keating \cite{AK} computed the first moment at the central point $1/2$ for the family $L(s,\chi_P)$, with a power saving error term. They also obtained the leading order term for the second moment, which has size $g^3$, and bounded the error term by $O(g^2)$. We improve their result and prove the following.

\begin{theorem}\label{mainthm}
For $q$ an odd number, we have
$$\frac{1}{|\mathcal{P}_{2g+1}|} \sum_{P\in \mathcal{P}_{2g+1}} L \big( \tfrac{1}{2},\chi_P \big)^2=\frac{g^3}{3\zeta_q(2)}+g^2 \Big( \frac{3}{2}+\frac{1}{2q}\Big) + O_\varepsilon(g^{3/2+\varepsilon}),$$ where the sum is over monic, irreducible polynomials with coefficients in $\mathbb{F}_q[x]$. 
\end{theorem}

We also prove the following.

\begin{theorem}\label{mainthm1}
Let $q$ be a prime power with $(q,6)=1$. Let $E/\mathbb{F}_q(t)$ be a fixed elliptic curve with discriminant $\Delta$ and $M$ be the product of the finite primes where $E$ has multiplicative reduction. Then for $g \geq \frac{ \deg(\Delta)-2}{2}$, we have 
\begin{align*}
\frac{1}{|\mathcal{P}_{2g+1}|} \sum_{P\in \mathcal{P}_{2g+1}} \epsilon^-L' ( E \otimes \chi_P, \tfrac{1}{2})&=2(\log q)\big(\mathcal{A}_E(1;1)-\epsilon_{2g+1} \epsilon(E)\mathcal{A}_E(M;1)\big)L(\emph{Sym}^2 E, 1)g\\
&\qquad\qquad + O_\varepsilon(g^{1/4+\varepsilon}),
\end{align*}
where $\epsilon^{-}$ is defined in \eqref{epsilon_minus} and \eqref{tfe}, and $\mathcal{A}_E(N;u)$ is given by \eqref{def_bm}. In particular, unless $\epsilon_{2g+1} \epsilon(E)=1$ and $M=1$, we obtain an asymptotic formula.
\end{theorem}

We remark that Andrade and Keating's approach would give an error term of size $O(g)$ for the above mean value, and hence fails to give an asymptotic formula.

Define the analytic rank of the twisted elliptic curve $E\otimes\chi_P$ by
\[
r_{E\otimes\chi_P}:=\text{ord}_{s=1/2}L(E\otimes\chi_P,s).
\]
From Theorem \ref{mainthm1} we obtain the following corollary.
\begin{corollary}
Unless $\epsilon_{2g+1} \epsilon(E)=1$ and $M=1$, there exists a monic irreducible polynomial $P$ of degree $2g+1$ such that
\[
r_{E\otimes\chi_P}=1.
\]
\end{corollary}

Computing moments in families of $L$--functions is a well studied problem, due to its applications to nonvanishing results, the subconvexity problem etc. Jutila \cite{jutila} computed the first moment in the family of quadratic Dirichlet $L$--functions, obtaining a power savings error term. Soundararajan \cite{sound} obtained asymptotics for the second and third moments for the family $L(s,\chi_{8d})$ for $d$ an odd, square-free, positive number. As a corollary, he showed that more than $87.5 \%$ of $L(1/2,\chi_d)$ do not vanish. Chowla conjectured that $L(1/2,\chi_d)$ is never equal to $0$.

Considering the family of quadratic Dirichlet $L$--functions with prime conductor, Jutila \cite{jutila} also showed that
$$ \sum_{\substack{ p \leq X \\ p \equiv 3 \pmod 4}} L (\tfrac{1}{2},\chi_p) = \frac{X \log X}{4}+ O_\varepsilon (X (\log X)^{\epsilon}).$$
This family is more difficult to work with due to the fact that the sums are over primes, as opposed to sums over essentially square-free numbers as in \cite{sound}. Conditionally on the Generalized Riemann Hypothesis (GRH), Baluyot and Pratt \cite{bp} obtained the leading order term in the asymptotic for the second moment. Specifically, they proved that
$$ \sum_{\substack{ p \leq X \\ p \equiv 1 \pmod 8}} L ( \tfrac{1}{2},\chi_{p})^2 = c X (\log X)^3 + O(X (\log X)^{11/4}),$$ for some explicit constant $c$. Unconditionally, they obtained upper and lower bounds of the right order of magnitude. Using sieve methods, they also showed that more than $9\%$ of $L(1/2,\chi_p)$ are non-zero. Under GRH, Andrade and Baluyot \cite{ab} computed the $1$--level density in the family and obtained that more than $75\%$ of the $L$--functions evaluated at the central point do not vanish.

The corresponding problem of computing moments in the family of quadratic Dirichlet $L$--functions with prime conductor over function fields was considered by Andrade and Keating \cite{AK}. For the second moment, they showed that
\begin{equation}
\frac{1}{|\mathcal{P}_{2g+1}|} \sum_{P\in \mathcal{P}_{2g+1}} L \big( \tfrac{1}{2},\chi_P \big)^2=\frac{g^3}{3\zeta_q(2)}+ O (g^2 ).
\label{ak_form}
\end{equation}
Similar to the recipe developed by Conrey, Farmer, Keating, Rubinstein and Snaith \cite{cfkrs}, Andrade, Jung and Shamesaldeen \cite{A2} conjectured asymptotic formulas for the integral moments of $L(1/2,\chi_P)$. Specifically, the conjecture is that
\begin{equation}
\frac{1}{|\mathcal{P}_{2g+1}|} \sum_{P\in \mathcal{P}_{2g+1}} L \big( \tfrac{1}{2},\chi_P \big)^k \sim P_k(2g+1),
\label{conj}
\end{equation} where $P_k$ is an explicit polynomial of degree $g(g+1)/2$.

To obtain the asymptotic formula \eqref{ak_form}, Andrade and Keating used the approximate functional equation and then computed a diagonal contribution from square polynomials, which gives the main term of size $g^3$. To bound the contribution from non-squares, they used the Weil bound (which follows from GRH over function fields). To explicitly compute the term of size $g^2$ in Theorem \ref{mainthm}, we are more careful in bounding the error term coming from non-square polynomials. After using the approximate functional equation, we truncate the Dirichlet series close to the endpoint. On the first, longer Dirichlet polynomial, we compute the diagonal term and bound the off-diagonal using the Weil bound. Since this Dirichlet polynomial is shorter than the one considered by Andrade and Keating, we obtain a saving on the error term. For the tail of the Dirichlet polynomial, we use the Perron formula and express it in terms of a shifted moment expression integrated along a circle around the origin. For the integral on a small arc around the origin, we use a recursive formula for the shifted moment and exhibit some explicit cancellation between this term and the diagonal. For the integral along the complement of the small arc, we use upper bounds for moments. A similar idea was used in the computation of lower order terms for the fourth moment of quadratic Dirichlet $L$--functions over function fields in \cite{AF4th}. 
Note that we expect off-diagonal terms to contribute to the coefficient of $g$ in the asymptotic formula \eqref{conj}. For Theorem \ref{mainthm}, both the $g^3$ and $g^2$ terms come from the diagonal. To explicitly compute off-diagonal terms, one would need to use a more refined method rather than relying on the Weil bound. 

In the orthogonal family of quadratic twists of a fixed modular form, the first moment was computed in \cite{BFH}, \cite{MM}, \cite{IW}. The second moment was considered by Soundararajan and Young \cite{soundyoung}, who obtained an asymptotic formula with the leading order term, conditionally on GRH. Unconditionally, they obtained a lower bound which matches the answer conjectured by Keating and Snaith in \cite{ks2}. Some of the work in the present paper is inspired by ideas used by Soundararajan and Young in \cite{soundyoung}. Using similar ideas, also under GRH, Petrow \cite{petrow} obtained several asymptotic formulas for moments of derivatives in this orthogonal family when the sign of the functional equation is equal to $-1$.

Similar problems over function fields were considered in \cite{BFKR}. The authors computed the first and second moments in the family of $L$--functions associated to quadratic twists of a fixed elliptic curve over $\mathbb{F}_q(t)$, and various other moments involving derivatives of these $L$--functions. These asymptotic formulas allow them to deduce lower bounds on the correlations between the analytic ranks of quadratic twists of two distinct elliptic curves. Note that in Theorem \ref{mainthm1} we compute the first moment for derivatives of the $L$--functions with root number equal to $-1$. Our methods do not allow us to obtain the mean value for the $L$--functions themselves, as the error term coming from using upper bounds for moments would dominate the diagonal term which has constant size in this case.

\medskip
{\bf Acknowledgements.} The authors would like to thank Kyle Pratt for useful discussions regarding his work with S. Baluyot on moments of Dirichlet $L$--functions with prime conductor \cite{bp}.

\section{Background}

Fix an odd number $q$. Let $\mathcal{M}$ denote the set of monic polynomials with coefficients in $\mathbb{F}_q$, and $\mathcal{M}_{\leq n}$ be the set of monic polynomials with degree less than or equal to $n$. Let $\mathcal{P}_n$ denote the set of monic, irreducible polynomials over $\mathbb{F}_q[x]$. The norm of a polynomial $f$ is defined to be $|f| =q^{\deg(f)}$.

The Prime Polynomial Theorem states that
\begin{equation}
| \mathcal{P}_n| = \frac{q^n}{n} + O \Big( \frac{q^{n/2}}{n} \Big).
\label{ppt}
\end{equation}

The quadratic character over $\mathbb{F}_q[t]$ is defined as follows. For $P$ a monic, irreducible polynomial and $f$ a monic polynomial, let
$$ \chi_P(f) = \Big(\frac{P}{f} \Big),$$ where $( \frac{P}{f})$ is the quadratic residue symbol over $\mathbb{F}_q[x]$.
\kommentar{$$ \Big( \frac{f}{P} \Big)= 
\begin{cases}
1 & \mbox{ if } P \nmid f, f \text{ is a square modulo }P, \\
-1 & \mbox{ if } P \nmid f, f \text{ is not a square modulo }P, \\
0 & \mbox{ if } P|f.
\end{cases}
$$
}

The zeta-function is defined as 
$$\zeta_q(s) = \sum_{f\in\mathcal{M}} \frac{1}{|f|^s},$$ for $\Re(s)>1$. Since there are $q^n$ monic polynomials of degree $n$, one can easily show that
$$\zeta_q(s) = \frac{1}{1-q^{1-s}},$$ which provides a meromorphic continuation of $\zeta_q$ with a simple pole at $s=1$. We will often make the change of variables $u=q^{-s}$, and then the zeta-function becomes
$$ \mathcal{Z}(u) = \zeta_q(s) = \sum_{f \in\mathcal{M}} u^{\deg(f)} = \frac{1}{1-qu},$$ with a simple pole at $u=1/q$. Note that $\mathcal{Z}(u)$ can also be written in terms of an Euler product as
$$ \mathcal{Z}(u) = \prod_Q \Big(1-u^{\deg(Q)}\Big)^{-1},$$ where the product is over monic, irreducible polynomials in $\mathbb{F}_q[t]$.

For $P$ a monic irreducible polynomial, the $L$--function associated to the quadratic character $\chi_P$ is defined by
$$L(s,\chi_P) = \sum_{f \in \mathcal{M}} \frac{\chi_P(f)}{|f|^s} = \prod_Q \bigg(1-\frac{\chi_P(Q)}{ |Q|^{s}}\bigg)^{-1}.$$ 
Similarly as before, with the change of variables $u=q^{-s}$, one has
$$ \mathcal{L}(u,\chi_P) = \sum_{f \in \mathcal{M}} \chi_P(f) u^{\deg(f)} = \prod_Q \Big(1-\chi_P(Q) u^{\deg(Q)}\Big)^{-1}.$$
By orthogonality of characters, it follows that $\mathcal{L}(u,\chi_P)$ is a polynomial of degree at most $\deg(P)-1$. 

For $P \in \mathcal{P}_{2g+1}$, the $L$--function satisfies the following functional equation
\begin{equation}
 \mathcal{L}(u,\chi_P) = (qu^2)^g \mathcal{L} \Big(  \frac{1}{qu},\chi_P \Big).
 \label{fe}
 \end{equation}

To define elliptic curve $L$--functions over function fields, we take $q$ to be a prime power with $(q,6)=1$. 
Let $E/\mathbb{F}_q(t)$ be an elliptic curve defined by $y^2=x^3+Ax+B$, with $A, B \in \mathbb{F}_q[t]$ and discriminant $\Delta = 4A^3+27B^2$ such that $\deg_t ( \Delta)$ is minimal among models of $E/\mathbb{F}_q(t)$ of this form. 
The normalized $L$--function associated to the elliptic curve $E/\mathbb{F}_q(t)$ has a Dirichlet series and an Euler product which converge for $\Re(s)>1$, as follows.
\begin{align}
L(E,s):=\mathcal{L}(E,u)& = \sum_{f \in \mathcal{M}} \lambda(f) u^{\deg(f)} \label{coefs} \\
&= \prod_{Q | \Delta} \Big(1- \lambda(Q) u^{\deg(Q)}\Big)^{-1} \prod_{Q \nmid \Delta} \Big(1- \lambda(Q) u^{\deg(Q)} + u^{2 \deg(Q)}\Big)^{-1}.\nonumber
\end{align}
One can show that the $L$--function is a polynomial in $u$ with integer coefficients and has degree
\begin{equation}
 \mathfrak{n} := \deg\big( \mathcal{L}(E,u)\big) = \deg(M) + 2 \deg(A) -4,\label{degree}
 \end{equation} where $M$ denotes the product of the finite primes where $E$ has multiplicative reduction and $A$ the product of the finite primes where $E$ has additive reduction (See \cite{hall}). The $L$--function satisfies a functional equation; namely, there exists $\epsilon(E) \in \{ \pm 1 \}$ such that
$$ \mathcal{L}(E,u) = \epsilon(E) (\sqrt{q} u)^{\mathfrak{n}} \mathcal{L} \Big( E, \frac{1}{qu}\,\Big).$$

For $P \in \mathcal{P}_{2g+1}$ and $(P, \Delta)=1$, we consider the twisted elliptic curve $E \otimes \chi_P $ having the affine model $y^2 = x^3+ P^2 Ax+ P^3B$. The $L$--function of the twisted elliptic curve has the following Dirichlet series and Euler product
\begin{align*} 
\mathcal{L}(E \otimes \chi_P,u) &= \sum_{f \in \mathcal{M}} \lambda(f) \chi_P(f) u^{\deg(f)} \\
&= \prod_{Q | \Delta}\Big (1- \lambda(Q) \chi_P(Q) u^{\deg(Q)}\Big)^{-1} \prod_{Q \nmid \Delta P}\Big (1- \lambda(Q) \chi_P (Q) u^{\deg(Q)} +u^{2 \deg(Q)}\Big)^{-1}.
\end{align*}
The $L$--function $\mathcal{L}(E \otimes \chi_P,u)$ is a polynomial of degree $\mathfrak{n}+ 2 \deg(P)$ and moreover it satisfies the functional equation
\begin{equation}\label{tfe}
\mathcal{L}(  E \otimes \chi_P, u ) = \epsilon\, (\sqrt{q} u)^{\mathfrak{n}+ 2 \deg(P)} \mathcal{L}\Big (  E \otimes \chi_P, \frac{1}{qu}\,\Big),
\end{equation}
where $\epsilon$ is the root number of the twisted elliptic curve and it is equal to
$$\epsilon= \epsilon_{\deg(P)} \epsilon(E)\chi_P(M).$$ In the above equation, $\epsilon_{\deg( P)} \in \{ \pm 1 \}$ is an integer which only depends on the degree of $P$ (see Proposition $4.3$ in \cite{hall}).
We set
\begin{equation}
\epsilon^{-} = \frac{1-\epsilon}{2}.
\label{epsilon_minus}
\end{equation}

 \section{Preliminary Lemmas}
 Here we will gather a few lemmas we need.
 We have the following approximate functional equation.
 \begin{lemma}\label{afe}
For $P \in \mathcal{P}_{2g+1}$, we have
 \begin{align*}
\mathcal{L} \Big( \frac{u}{\sqrt{q}},\chi_P \Big)^2 &= \sum_{f \in \mathcal{M}_{\leq 2g}} \frac{\tau(f)\chi_P(f)u^{\deg(f)} }{\sqrt{|f|}}+u^{4g} \sum_{f \in \mathcal{M}_{\leq 2g-1}} \frac{\tau(f)\chi_P(f) }{\sqrt{|f|}u^{\deg(f)}},
 \end{align*}
 where $\tau(f) = \sum_{f_1f_2=f} 1$ is the divisor function.
 \end{lemma}
 \begin{proof}
 See \cite{AK}, equation $(4.4)$. 
 \end{proof}
 
  \begin{lemma}\label{afe1}
For $P \in \mathcal{P}_{2g+1}$ and $E/\mathbb{F}_q(t)$ such that $\epsilon=-1$, we have
$$ L' ( E \otimes \chi_P, \tfrac{1}{2}) =2(\log q)  \sum_{\deg(f)\leq [\mathfrak{n}/2]+2g+1} \frac{\big([\mathfrak{n}/2]+2g+1-\deg(f)\big)\lambda(f) \chi_P(f)}{\sqrt{|f|}}.$$
 \end{lemma}
 \begin{proof}
 See Lemma $2.3$ in \cite{BFKR}. 
 \end{proof}
 The following Weil bound holds for character sums over primes.
  \begin{lemma}\label{PVlemma}
 For $f$ not a square, we have
 $$\frac{1}{|\mathcal{P}_{2g+1}|} \sum_{P \in \mathcal{P}_{2g+1}} \chi_P(f) \ll q^{-g} \deg(f).$$ 
 \end{lemma}
  \begin{proof}
 See \cite{R}, equation $(2.5)$. 
 \end{proof}

   \section{Upper bounds}
In this section we will prove the following upper bounds for moments, whose proofs are similar to the proof of the upper bound for moments of the Riemann zeta-function in \cite{soundupperbounds}. A similar function field proof can also be found in \cite{AF4th}.
 \begin{proposition}\label{upperbounds}
 Let $u=e^{i \theta}$ with $\theta \in [0,2\pi)$. Then for every $k>0$ and $\epsilon>0$, we have
 $$ \frac{1}{|\mathcal{P}_{2g+1}|}\sum_{P \in \mathcal{P}_{2g+1}} \Big|\mathcal{L} \Big(\frac{u}{\sqrt{q}},\chi_P \Big) \Big|^k \ll_\varepsilon g^{\varepsilon} \exp \Big( k \mathcal{M}_1(u,g)+\frac{k^2}{2}\mathcal{V}_1(u,g) \Big),$$ where
 $$\mathcal{M}_1(u,g) = \frac{1}{2} \log\min \Big\{g,\frac{1}{ \overline{2\theta}}\Big\}$$ and
 $$\mathcal{V}_1(u,g) = \mathcal{M}_1(u,g)+\frac{\log g}{2}.$$
 Here for $\theta \in [0, 2\pi)$ we denote $\overline{\theta} = \min \{ \theta, 2 \pi - \theta \}$.
 \end{proposition}
 
  \begin{proposition}\label{upperbounds1}
 Let $u=e^{i \theta}$ with $\theta \in [0,2\pi)$ and let $m = \deg \big(\mathcal{L}(E \otimes \chi_{P},v)\big)$. Then for every $k>0$ we have
 $$ \frac{1}{|\mathcal{P}_{2g+1}|}\sum_{P \in \mathcal{P}_{2g+1}} \Big|\mathcal{L} \Big(E\otimes \chi_P,\frac{u}{\sqrt{q}}\Big) \Big|^k \ll_\varepsilon g^{\varepsilon} \exp \Big( k \mathcal{M}_2(u,m)+\frac{k^2}{2}\mathcal{V}_2(u,m) \Big),$$ where
 $$\mathcal{M}_2(u,m) = -\frac{1}{2} \log \min  \Big \{ m ,\frac{1}{ \overline{2\theta}} \Big\}$$ and
 $$\mathcal{V}_2(u,m) = -\mathcal{M}_2(u,m)+\frac{\log m}{2}.$$
 \end{proposition}
 
Propositions \ref{upperbounds} and  \ref{upperbounds1} immediately lead to the following corollary.
\begin{corollary}\label{corollaryupper}
Let $u=e^{i \theta}$ with $\theta \in [0,2\pi)$. Then
 \begin{equation*}
 \frac{1}{|\mathcal{P}_{2g+1}|}  \sum_{P \in \mathcal{P}_{2g+1}} \Big|\mathcal{L} \Big(\frac{u}{\sqrt{q}},\chi_P \Big) \Big|^2 \ll_\varepsilon g^{1+\epsilon} \min \Big\{ g,\frac{1}{ \overline{2\theta}}\Big\}^{2}
 \end{equation*}
 and
 \begin{equation*}
 \frac{1}{|\mathcal{P}_{2g+1}|}  \sum_{P \in \mathcal{P}_{2g+1}} \Big|\mathcal{L} \Big(E\otimes \chi_P,\frac{u}{\sqrt{q}}\Big) \Big| \ll_\varepsilon g^{1/4+\epsilon} \min \Big\{ g,\frac{1}{  \overline{2\theta}}\Big\}^{-1/4}.
 \end{equation*}
 \end{corollary}
 
Before proving the above propositions, we first need the following lemma.
 
 \begin{lemma}\label{dir_mom}
Let $h,l$ be integers such that $hl \leq g$ and $h>1$. For any complex numbers $a(Q)$ we have
$$\frac{1}{|\mathcal{P}_{2g+1}|} \sum_{P \in \mathcal{P}_{2g+1}} \bigg| \sum_{Q\in\mathcal{P}_{ \leq h}} \frac{ a(Q)\chi_P(Q)}{\sqrt{|Q|}} \bigg|^{2l} \ll \frac{ (2l)!}{l! 2^l} \bigg(\sum_{Q\in\mathcal{P}_{ \leq h}} \frac{ |a(Q)|^2}{|Q|} \bigg)^l.$$
\end{lemma}
\begin{proof}
The proof is similar to the proof of Lemma 6.3 in \cite{soundyoung}. Expanding out and using Lemma \ref{PVlemma} we have
\begin{align}
&\frac{1}{|\mathcal{P}_{2g+1}|} \sum_{P \in \mathcal{P}_{2g+1}} \bigg| \sum_{Q\in\mathcal{P}_{ \leq h}} \frac{a(Q)\chi_P(Q) }{\sqrt{|Q|}} \bigg|^{2l} \nonumber \\
&\qquad\ll \sum_{\substack{Q_1,\ldots,Q_{2l}\in\mathcal{P}_{ \leq h}\\Q_1\ldots Q_{2l}=\square}} \frac{|a(Q_1)\ldots a(Q_{2l})|}{\sqrt{|Q_1\ldots Q_{2l}|}}+O\bigg(hlq^{-g}\sum_{Q_1,\ldots,Q_{2l}\in\mathcal{P}_{ \leq h}} \frac{|a(Q_1)\ldots a(Q_{2l})}{\sqrt{|Q_1\ldots Q_{2l}|}}\bigg). \label{pv}
\end{align}
For the first term, we note that $Q_1\ldots Q_{2l}=\square$ if and only if there is a way to pair up
the indices so that the corresponding polynomials are equal. As there are $(2l)!/l!2^l$ ways to pair up $2l$
indices, it follows that
\begin{align*}
\sum_{\substack{Q_1,\ldots,Q_{2l}\in\mathcal{P}_{ \leq h}\\Q_1\ldots Q_{2l}=\square}} \frac{|a(Q_1)\ldots a(Q_{2l})|}{\sqrt{|Q_1\ldots Q_{2l}|}}\ll \frac{(2l)!}{l!2^l} \bigg(\sum_{Q\in\mathcal{P}_{ \leq h}} \frac{ |a(Q)|^2}{|Q|} \bigg)^l.
\end{align*}
For the second term in \eqref{pv}, we use the Cauchy-Schwarz inequality to see that it is bounded by
\begin{align*}
&\ll gq^{-g}\bigg( \sum_{Q\in\mathcal{P}_{ \leq h}} \frac{a(Q)}{\sqrt{|Q|}} \bigg)^{2l} \ll gq^{-g}\bigg( \sum_{Q\in\mathcal{P}_{ \leq h}} \frac{|a(Q)|^2}{|Q|} \bigg)^{l}\bigg( \sum_{Q\in\mathcal{P}_{ \leq h}} 1 \bigg)^{l}.
\end{align*}
Since $h>1$ and $hl \leq g$, using the Prime Polynomial Theorem, it follows that the above is bounded by
$$\ll \bigg( \sum_{Q\in\mathcal{P}_{ \leq h}} \frac{|a(Q)|^2}{|Q|} \bigg)^{l},$$
and the proof is complete.

\end{proof}

We shall only illustrate the proof of Proposition \ref{upperbounds1}. The proof of Proposition \ref{upperbounds} follows along the same lines and it is also similar to the proof of Theorem $2.7$ in \cite{AF4th}, using Lemma \ref{dir_mom} instead of Lemma $8.4$ in \cite{AF4th}.

\begin{lemma} \label{ub_n}
Let $$N(V,u) =\frac{1}{|\mathcal{P}_{2g+1}|} \bigg| \Big\{P \in \mathcal{P}_{2g+1}: \log \Big| \mathcal{L}\Big(E \otimes \chi_{P}, \frac{u}{\sqrt{q}}\Big)\Big| \geq V + \mathcal{M}_2(u,m) \Big\} \bigg|.$$
If $\sqrt{\log m} \leq V \leq \mathcal{V}_2(u,m)$, then
$$N(V,u) \ll \exp \bigg( -\frac{V^2}{2 \mathcal{V}_2(u,m) } \Big(1-\frac{8}{\log \log m} \Big)\bigg) ; $$
if $ \mathcal{V}_2(u,m)< V \leq  \frac{\log \log m}{16}\mathcal{V}_2(u,m)$, then
$$ N(V,u) \ll \exp \bigg( -\frac{V^2}{2 \mathcal{V}_2(u,m) } \Big( 1- \frac{8V}{ \mathcal{V}_2(u,m) \log \log m} \Big)^2 \bigg);$$
and if $V > \frac{\log \log m}{16}\mathcal{V}_2(u,m)$, then
$$N(V,u) \ll \exp \Big(-\frac{V \log V}{4500} \Big).$$
\end{lemma}
\begin{proof}
Let 
$$ A= 
\begin{cases}
\displaystyle \frac{\log \log m}{2} & \mbox{ if } \sqrt{\log m} \leq V \leq \mathcal{V}_2(u,m), \\
\displaystyle \frac{\log \log m}{2V}\mathcal{V}_2(u,m)  & \mbox{ if } \mathcal{V}_2(u,m) < V \leq \frac{ \log \log m}{16}\mathcal{V}_2(u,m), \\
 8 & \mbox{ if } V> \frac{ \log \log m}{16}\mathcal{V}_2(u,m),
\end{cases},
$$
and
$$\frac{m}{h}=\frac{V}{A}.$$ Proposition 4.3 in \cite{BFKR} yields
\begin{align}\label{ineq_prop}
& \log \bigg| \mathcal{L}\Big(E \otimes \chi_{P},  \frac{u}{\sqrt{q}}\Big) \bigg| \leq \frac{m}{h}+\frac{1}{h} \Re \bigg(  \sum_{\substack{ j \geq 1 \\ \deg(Q^j) \leq h}} \frac{ ( \alpha(Q)^j + \beta(Q)^j)\chi_{P}(Q^j)\log q^{h-j \deg(Q)}}{|Q|^{(1/2+(1/h-i\theta)/ \log q)j}\log q^j} \bigg).
\end{align}

The contribution of the terms with $j \geq 3$ is bounded by $O(1)$. The terms with $j=2$ will contribute
\begin{align*}
\frac{1}{2h } \sum_{\deg(Q) \leq h/2} \frac{ (h - 2\deg(Q))(\lambda (Q^2)-1)\cos(2 \theta \deg(Q)) }{|Q|e^{2\deg(Q)/h}}.
\end{align*}
Let 
$$ F(h, \theta) = \sum_{n=1}^h \frac{ \cos(2n \theta)}{n e^{n/h}}.$$
As in Lemma 9.1 in \cite{AF4th}, we can show that
\begin{equation*}
F (h, \theta) =  \log \min  \Big \{ h ,\frac{1}{ \overline{2 \theta}} \Big\} +O(1).
\end{equation*}
Since
\begin{equation}
 \sum_{\deg(Q) \leq h} \frac{ \lambda(Q^2) \cos(2 \theta \deg(Q))}{|Q|e^{\deg(Q)/h}} = O(\log \log h),
 \label{ub_pf}
 \end{equation}
it follows that the contribution from $j=2$ is
\begin{align}
-\frac{1}{2 h} & \sum_{\deg(Q) \leq \frac{h}{2}} \frac{(h - 2 \deg(Q))\cos(2 \theta \deg(Q))}{|Q|e^{2\deg(Q)/h}}+O( \log \log m) \nonumber \\
&=-\frac{F(h/2,\theta)}{2} + O(\log \log m)\leq  -\frac{F(m, \theta)}{2} +\frac{m}{h}+ O(\log \log m) \nonumber\\
&= \mathcal{M}_2(u,m) + \frac{m}{h} + O( \log \log m). \label{aux}
\end{align}
Note that in the second line of the equation above we used the fact that
\begin{align*}
F(m,\theta) - F\Big(\frac h2,\theta\Big) &= \sum_{n=1}^m \frac{ \cos(2n \theta)}{n e^{n/m}}- \sum_{n=1}^{h/2} \frac{\cos(2n \theta)}{ne^{2n/h}},
\end{align*} and since $e^{-x} = 1+O(x)$, we have
\begin{align*}
F(m,\theta) - F\Big(\frac h2,\theta\Big)= \sum_{n=h/2+1}^m \frac{\cos(2n \theta)}{n} +O(1) \leq \frac{2m}{h} + O(1).
\end{align*}
Applying equation \eqref{aux} to \eqref{ineq_prop} hence leads to
\begin{align*}
 \log \bigg| \mathcal{L}\Big(E \otimes \chi_{P},  \frac{u}{\sqrt{q}}\Big)\bigg| &\leq \mathcal{M}_2(u,m) +\frac{3m}{h}\\
 &\qquad\qquad+ \frac{1}{h} \sum_{\deg(Q) \leq h} \frac{  (h - \deg(Q))\lambda(Q)\chi_{P}(Q)\cos( \theta \deg(Q)) }{ \sqrt{|Q|}e^{\deg(Q)/h}} .
\end{align*}

Let $S_1$ be the sum above truncated at $$\deg(Q) \leq h_0=\frac{h}{\log m}$$ and $S_2$ be the sum over primes with $h_0 < \deg(Q) \leq h$. If $P$ is such that $$ \log \bigg| \mathcal{L}\Big(E \otimes \chi_{P},  \frac{u}{\sqrt{q}}\Big)\bigg| \geq V+\mathcal{M}_2(u,m),$$ then $$S_1 \geq V_1:= V \Big( 1-\frac{4}{A}\Big) \qquad\text{or}\qquad S_2 \geq \frac{V}{A}.$$ 
Let $$\mathcal{F}_1 = \{P \in \mathcal{P}_{2g+1} : S_1 \geq V_1\}\qquad \text{and} \qquad\mathcal{F}_2 = \{ P \in \mathcal{P}_{2g+1} : S_2 \geq V/A \}.$$

If $P \in \mathcal{F}_2$, then by Markov's inequality and Lemma \ref{dir_mom} it follows that
\begin{align*} |\mathcal{F}_2|&\ll\Big(\frac{A}{V}\Big)^{2l}\sum_{P \in \mathcal{P}_{2g+1}} \bigg(\sum_{h_0<\deg(Q) \leq h} \frac{ a(Q)\chi_P(Q)}{\sqrt{|Q}|}\bigg)^{2l}\\
& \ll |\mathcal{P}_{2g+1}| \Big( \frac{A}{V}\Big)^{2l}  \frac{ (2l)!}{l! 2^l}\Big( \sum_{h_0<\deg(Q) \leq h} \frac{ |a(Q)|^2}{|Q|} \Big)^l,\end{align*} 
for any $l \leq g/h$ where $$ a(Q) =\frac{  (h- \deg(Q)) \lambda(Q)\cos(\theta \deg(Q))}{h e^{\deg(Q)/h}}.$$
We pick $l = [g/h]$ and note that $a(P) \ll 1$ and $m=4g+O(1)$. Hence we get that
\begin{equation}
|  \mathcal{F}_2 | \ll |\mathcal{P}_{2g+1}| \Big( \frac{A}{V} \Big)^{2l} \Big(\frac{2l}{e} \Big)^{l} (\log \log m)^l \ll |\mathcal{P}_{2g+1}| \exp \Big(-\frac{V\log V}{8A}  \Big).
\label{f2_bd}
\end{equation}

If $P \in \mathcal{F}_1$ then similarly for any $l \leq g/h_0$, we have
$$ | \mathcal{F}_1| \ll |\mathcal{P}_{2g+1}|\frac{1}{V_1^{2l}} \frac{ (2l)!}{l! 2^l} \Big( \sum_{\deg(Q) \leq h_0} \frac{ |a(Q)|^2}{|Q|} \Big)^l.$$
Using the expression for $a(Q)$ and equation \eqref{ub_pf} we obtain 
$$  | \mathcal{F}_1| \ll |\mathcal{P}_{2g+1}| \Big(\frac{2l}{eV_1^2} \Big)^{l} \big( \mathcal{V}_2(u,m) + O(\log \log m) \big)^l.$$
 If $V \leq \mathcal{V}_2(u,m)^2$, 
  then we pick $l = [ V_1^2/2 \mathcal{V}_2(u,m)]$, and if $V > \mathcal{V}_2(u,m)^2$, 
   then we pick $l=[10V]$. In doing so we get 
 \begin{equation}
 | \mathcal{F}_1| \ll |\mathcal{P}_{2g+1}| \exp \Big( - \frac{V_1^2}{2 \mathcal{V}_2(u,m)} \Big)+ |\mathcal{P}_{2g+1}| \exp ( -V \log V).
 \label{f1_bd}
 \end{equation}
 Combining the bounds \eqref{f1_bd} and \eqref{f2_bd} finishes the proof of Lemma \ref{ub_n}.
\end{proof}

\begin{proof}[Proof of Theorem \ref{upperbounds1}]
We have the following.
\begin{align*}
 \frac{1}{|\mathcal{P}_{2g+1}|}\sum_{P \in \mathcal{P}_{2g+1}} \Big|\mathcal{L} \Big(E\otimes \chi_P,\frac{u}{\sqrt{q}}\Big) \Big|^k& = - \int_{-\infty}^{\infty} \exp \big(kV + k \mathcal{M}_2(u,m)\big) dN(V,u) \\
&= k \int_{-\infty}^{\infty} \exp \big(kV+k \mathcal{M}_2(u,m)\big) N(V,u) dV.
\end{align*}
We apply Lemma \ref{ub_n} in the form
$$N(V,u) \ll_\varepsilon 
\begin{cases}
m^{\varepsilon} \exp \Big(-\frac{V^2}{2 \mathcal{V}_2(u,m)} \Big) & \mbox{ if } V \leq 8k \mathcal{V}_2(u,m), \\
m^{\varepsilon} \exp(-4kV) & \mbox{ if }V > 8k \mathcal{V}_2(u,m)
\end{cases}
$$
to the above formula and finish the proof of the theorem.
\end{proof}

 \section{Main propositions}
 
 \begin{proposition}\label{mainprop}
For $X< g$, let
 \[
E_1(X)=\frac{1}{|\mathcal{P}_{2g+1}|} \sum_{P \in \mathcal{P}_{2g+1}} \sum_{\substack{2 X<\deg(f)\leq 2g}} \frac{\tau(f)\chi_P(f) }{\sqrt{|f|}}. \]
 Then
 \begin{align*}
 E_1(X)&=\frac{g^2(g-X)}{2\zeta_q(2)}+ O_\varepsilon \big( g^{3/2+\varepsilon}(g-X)\big)+ O \big( g^{1/2}(g-X)^3\big).
 \end{align*}
  \end{proposition}
 \begin{proof}
 Applying the Perron formula we get
  \begin{align*}
E_1(X) &= \frac{1}{2 \pi i} \oint_{|u|=1}\frac{1}{|\mathcal{P}_{2g+1}|}\sum_{P \in \mathcal{P}_{2g+1}}  \mathcal{L} \Big( \frac{u}{\sqrt{q}}, \chi_P \Big)^2 \, \frac{(1-u^{2g-2X})\,du}{u^{2g+1}(1-u)}\nonumber\\
&= \frac{1}{2 \pi i}\int_{C_{1}} \frac{1}{|\mathcal{P}_{2g+1}|}\sum_{P \in \mathcal{P}_{2g+1}}  \mathcal{L} \Big( \frac{u}{\sqrt{q}}, \chi_P \Big)^2 \, \frac{(1-u^{2g-2X})\,du}{u^{2g+1}(1-u)} \nonumber \\
&\qquad\qquad+\ \frac{1}{2 \pi i}\int_{C_{2}} \frac{1}{|\mathcal{P}_{2g+1}|}\sum_{P \in \mathcal{P}_{2g+1}}  \mathcal{L} \Big( \frac{u}{\sqrt{q}}, \chi_P \Big)^2 \, \frac{(1-u^{2g-2X})\,du}{u^{2g+1}(1-u)},
 \end{align*}
 where $C_{1}$ denotes the arc of angle $4\pi\theta_1$ centered around $1$,  with $1/g\ll \theta_1=o(1)$, and $C_{2}$ is its complement. Let $E_{11}$ denote the integral over $C_1$ and $E_{12}$ the integral over $C_2$.
Note that 
\[
\Big|\frac{1-u^{2g-2X}}{1-u}\Big|\leq 2 (g-X),
\]
so there is no pole at $u = 1$. On $C_{2}$ using Corollary \ref{corollaryupper} we obtain
\begin{equation}\label{E2}
E_{12}\ll_\varepsilon g^{1+\varepsilon} (g-X)\theta_1^{-1}.
\end{equation}

On $C_{1}$ we use the approximate functional equation in Lemma \ref{afe}. For $f \neq \square$ we apply Lemma \ref{PVlemma}. In doing so we get
 \begin{align*}  \frac{1}{|\mathcal{P}_{2g+1}|}\sum_{P \in \mathcal{P}_{2g+1}}  \mathcal{L} \Big( \frac{u}{\sqrt{q}}, \chi_P \Big)^2  &=  \sum_{f \in \mathcal{M}_{\leq g}} \frac{ \tau(f^2)u^{2 \deg(f)}}{|f|}  + u^{4g} \sum_{f \in \mathcal{M}_{\leq g-1}} \frac{ \tau(f^2)}{|f| u^{2 \deg(f)}}+O(g^2),
 \end{align*}
 and hence
 \begin{align*}
 E_{11}&=\sum_{f \in \mathcal{M}_{\leq g}} \frac{ \tau(f^2)}{|f|}\frac{1}{2 \pi i}\int_{C_{1}}\frac{(1-u^{2g-2X})\,du}{u^{2g-2 \deg(f)+1}(1-u)}\\
 &\qquad\qquad+ \sum_{f \in \mathcal{M}_{\leq g-1}} \frac{ \tau(f^2)}{|f|} \frac{1}{2 \pi i}\int_{C_{1}}\frac{(1-u^{2g-2X})\,du}{u^{-2g+2 \deg(f)+1}(1-u)}+O(g^2(g-X)\theta_1).
 \end{align*}
Now
 \begin{equation}\label{series} \sum_{f \in \mathcal{M}} \tau(f^2) v^{\deg(f)} = \frac{ \mathcal{Z}(v)^3}{\mathcal{Z}(v^2)},\end{equation}
so using the Perron formula and making a change of variables in the second equation, we have
 \begin{align*}
E_{11}&=\frac{1}{(2 \pi i)^2}\oint_{|v|=r}\int_{C_{1}}\frac{ \mathcal{Z}(u^2v/q)^3 (1-u^{2g-2X})}{\mathcal{Z}(u^4v^2/q^2) (1-u)(1-v) u^{2g+1} v^{g+1}} \, du \, dv\\
 &\qquad+\frac{1}{(2 \pi i)^2}\oint_{|v|=r}\int_{C_{1}}\frac{ \mathcal{Z}(v/(u^2q))^3(1-u^{2g-2X})}{\mathcal{Z}(v^2/(u^4q^2)) (1-u) (1-v)u^{-2g+1}v^{g}} \, du \, dv+O(g^2(g-X)\theta_1)\\
&=\frac{1}{(2 \pi i)^2}\oint_{|v|=r}\int_{C_{1}}\frac{(1-u^{2g-2X})(1-v^2/q)}{uv^{g+1}(1-u)(1-v/u^2)(1-v)^3} \,du \, dv \\
 &\qquad+\frac{1}{(2 \pi i)^2}\oint_{|v|=r}\int_{C_{1}}\frac{u(1-u^{2g-2X})(1-v^2/q)}{v^{g}(1-u)(1-u^2v)(1-v)^3} \,du \, dv+O(g^2(g-X)\theta_1),
 \end{align*}
for $r<1$. We have
 \begin{align*} 
 \frac{1}{2 \pi i} \int_{C_{1}} &\frac{(1-u^{2g-2X})du}{u(1-u) (1-v/u^2)}= \sum_{j=0}^{2g-2X-1} \sum_{n=0}^{\infty}v^n \frac{1}{2 \pi i} \int_{C_{1}} u^{j-2n-1} \, du \\
 &= \sum_{j=0}^{2g-2X-1} \sum_{n=0}^{\infty}v^n \int_{-\theta_1}^{\theta_1} e^{2 \pi i \theta(j-2n)} \, d \theta = \frac{1}{ \pi } \sum_{j=0}^{2g-2X-1} \sum_{n=0}^{\infty}v^n \frac{ \sin(2 \pi (2n-j) \theta_1)}{2n-j},
 \end{align*}
 and similarly 
 $$ \frac{1}{2 \pi i} \int_{C_{1}} \frac{uv(1-u^{2g-2X}) du}{(1-u) (1-u^2v)} = \frac{1}{ \pi } \sum_{j=0}^{2g-2X-1} \sum_{n=1}^{\infty}v^n \frac{ \sin(2 \pi (2n+j) \theta_1)}{2n+j}.$$
Hence
 \begin{align*}
 E_{11} &= \frac{1}{\pi} \frac{1}{2 \pi i} \oint_{|v|=r} \sum_{j=0}^{2g-2X-1}\bigg( \sum_{n=0}^{g}v^n \frac{ \sin(2 \pi (2n-j) \theta_1)}{2n-j}+  \sum_{n=1}^{g} v^n \frac{ \sin(2 \pi (2n+j) \theta_1)}{2n+j} \bigg)\\
&\qquad\qquad\times \frac{(1-v^2/q)\,dv}{v^{g+1}(1-v)^3}+ O(g^2(g-X) \theta_1).
 \end{align*}
Enlarging the contour and evaluating the residue at $v=1$ we get that
 \begin{align*}
 E_{11} &= \frac{1}{2 \pi} \sum_{j=0}^{2g-2X-1} \bigg(\sum_{n=0}^{g}  \frac{ \sin(2 \pi (2n-j) \theta_1)}{2n-j} P(n)
 + \sum_{n=1}^{g}  \frac{ \sin(2 \pi (2n+j) \theta_1)}{2n+j}P(n) \bigg)\\
 &\qquad\qquad+O(g^2(g-X) \theta_1),
 \end{align*}
 where 
\begin{align*}
P(x)& = \Big(1-\frac{1}{q}\Big)(g-x)^2+\Big(3+\frac{1}{q}\Big)(g-x)+2\\
&=\Big(1-\frac{1}{q}\Big)x^2-\bigg(2\Big(1-\frac{1}{q}\Big)g+3+\frac{1}{q}\bigg)x+\Big(1-\frac{1}{q}\Big)g^2+\Big(3+\frac{1}{q}\Big)g+2.
\end{align*}
 Using Lemma $9.4$ in \cite{AF4th} we then obtain
 \begin{equation}
 E_{11} = \frac{g^2(g-X)}{2\zeta_q(2)}+ O \big( g^2(g-X) \theta_1\big)+O\big(g(g-X)^3\theta_1\big)+O\big(g(g-X)\theta_1^{-1}\big).
 \label{E1}
 \end{equation}

Combining equations \eqref{E2}, \eqref{E1} and choosing $\theta_1=1/\sqrt{g}$ we obtain the proposition.
 \end{proof}
 
  \begin{proposition}\label{mainprop1}
For $N$ a fixed square-free monic polynomial, $n\in\mathbb{N}$ fixed and $X<2g$, let
\begin{equation}\label{truncation2}
E_{2}(N,X,n)=\frac{1}{|\mathcal{P}_{2g+1}|} \sum_{P \in \mathcal{P}_{2g+1}}\sum_{X<\deg(f)\leq 2g+n} \frac{\big(2g+n-\deg(f)\big) \lambda(f) \chi_P(Nf)}{ \sqrt{|f|}}.
\end{equation}
 Then
 \begin{align*}
 E_2(N,X,n)&\ll_\varepsilon g^{1/4+\varepsilon}(2g-X)^2.
 \end{align*}
  \end{proposition}
  \begin{proof}
  Using the Perron formula for the sum over $f$ in \eqref{truncation2} we have
\begin{align*}
E_{2}(N,X,n)& =  \frac{1}{2 \pi i} \oint_{|u|=1}\frac{1}{|\mathcal{P}_{2g+1}|} \sum_{P \in \mathcal{P}_{2g+1}} \chi_P(N) \mathcal{L} \Big( E \otimes \chi_P, \frac{u}{\sqrt{q}} \Big) \\
&\qquad\qquad\times \Big( \frac{1-u^{2g+n-X}}{u^{2g+n}}-\frac{(2g+n-X)(1-u)}{u^{X+1}}\Big) \, \frac{du}{(1-u)^2}.
\end{align*}
Note that
\[
\Big( \frac{1-u^{2g+n-X}}{u^{2g+n}}-\frac{(2g+n-X)(1-u)}{u^{X+1}}\Big) \, \frac{1}{(1-u)^2}\ll (2g-X)^2,
\]
and so, in particular, there is no pole at $u=1$. The proposition hence follows after applying Corollary \ref{corollaryupper}.
  \end{proof}

 \section{Proof of Theorem \ref{mainthm}}

We first use the approximate functional equation in Lemma \ref{afe} to write
\begin{align*}
&\frac{1}{|\mathcal{P}_{2g+1}|} \sum_{P\in \mathcal{P}_{2g+1}} L \big( \tfrac{1}{2},\chi_P \big)^2=S_1+S_2,
\end{align*}
where
\[
S_1=\frac{1}{|\mathcal{P}_{2g+1}|} \sum_{P\in \mathcal{P}_{2g+1}}\sum_{f \in \mathcal{M}_{\leq 2g}} \frac{\tau(f)\chi_P(f)}{\sqrt{|f|}}
\]
and $S_2$ has a similar expression with $\mathcal{M}_{\leq 2g}$ being replaced by $\mathcal{M}_{\leq 2g-1}$. From Proposition \ref{mainprop} we obtain
\begin{align*}
S_1&=\frac{1}{|\mathcal{P}_{2g+1}|} \sum_{P\in \mathcal{P}_{2g+1}}\sum_{f \in \mathcal{M}_{\leq 2X}} \frac{\tau(f)\chi_P(f)}{\sqrt{|f|}}+\frac{g^2(g-X)}{2\zeta_q(2)}\\
&\qquad\qquad+ O_\varepsilon \big( g^{3/2+\varepsilon}(g-X)\big)+ O \big( g^{1/2}(g-X)^3\big).
\end{align*}
For $f\ne\square$ we apply Lemma \ref{PVlemma}. In doing so we get
 \[
 S_1=\sum_{f \in \mathcal{M}_{\leq X}} \frac{\tau(f^2)}{|f|}+\frac{g^2(g-X)}{2\zeta_q(2)}+ O_\varepsilon \big( g^{3/2+\varepsilon}(g-X)\big)+ O \big( g^{1/2}(g-X)^3\big)+O\big(q^{-g+X}g^2\big).
 \]

In view of \eqref{series} and the Perron formula, the above sum over $f$ is
 \begin{align*}
  \frac{1}{2\pi i}\oint_{|u|=r}\frac{ \mathcal{Z}(u/q)^3}{\mathcal{Z}(u^2/q^2)}\frac{du}{u^{X+1}(1-u)}=\frac{1}{2\pi i}\oint_{|u|=r}\frac{(1-u^2/q)du}{u^{X+1}(1-u)^4},
 \end{align*}
for $r<1$. By enlarging the contour of integration, passing the pole at $u=1$, we see that this is equal to
 \begin{equation*}
\frac{X^3}{6\zeta_q(2)}+X^2 +O(X).
 \end{equation*}
Thus
 \begin{equation*}
S_1 = \frac{g^3}{6 \zeta_q(2)}+g^2+ O_\varepsilon \big( g^{3/2+\varepsilon}(g-X)\big)+ O \big( g^{1/2}(g-X)^3\big)+O\big(q^{-g+X}g^2\big).
 \end{equation*} 

A similar computation leads to
$$S_2 = \frac{g^3}{6 \zeta_q(2)}+g^2 \Big(\frac{1}{2}+\frac{1}{2q}\Big)+ O_\varepsilon \big( g^{3/2+\varepsilon}(g-X)\big)+ O \big( g^{1/2}(g-X)^3\big)+O\big(q^{-g+X}g^2\big).$$

  To obtain Theorem \ref{mainthm}, we choose $X=g-[100 \log g]$.

 \section{Proof of Theorem \ref{mainthm1}}
 
 For $N$ a fixed square-free monic polynomial, let $T_E(N,X)=T_E(N,X;0)$, where
\begin{equation}\label{T_E}
T_E(N,X;\alpha) =\frac{1}{|\mathcal{P}_{2g+1}|} \sum_{P\in \mathcal{P}_{2g+1}} \sum_{f \in \mathcal{M}_{\leq X}} \frac{\lambda(f) \chi_P(Nf)}{|f|^{1/2+\alpha}}
\end{equation} 
and $|\alpha|\ll 1/g$. From Lemma \ref{afe1} we have
\begin{align*}
\epsilon^- L' ( E \otimes \chi_P, \tfrac{1}{2}) &=(\log q)\big(1-\epsilon_{2g+1}\epsilon(E)\chi_P(M)\big)\\
&\qquad\qquad\times  \sum_{\deg(f)\leq [\mathfrak{n}/2]+2g+1} \frac{\big([\mathfrak{n}/2]+2g+1-\deg(f)\big)\lambda(f) \chi_P(f)}{\sqrt{|f|}}.\end{align*}
Truncating the above sum at $\deg(f)\leq X$ and applying Proposition \ref{mainprop1} we get
\begin{align}\label{toevaluate}
&\frac{1}{|\mathcal{P}_{2g+1}|} \sum_{P\in \mathcal{P}_{2g+1}} \epsilon^{-} L'(E \otimes \chi_P,\tfrac{1}{2}) \nonumber\\
&\qquad\quad= (\log q) \big( [\mathfrak{n}/2]+2g+1\big) \Big(T_E(1, X)-\epsilon_{2g+1} \epsilon(E)T_E(M, X)\Big) \\
&\qquad\qquad\qquad+ \frac{\partial}{\partial \alpha}\Big(T_E(1,X;\alpha)-\epsilon_{2g+1} \epsilon(E) T_E(M,X;\alpha)\Big) \bigg|_{\alpha=0}+O_\varepsilon\big(g^{1/4+\varepsilon}(2g-X)^2\big).\nonumber
\end{align}
Thus we are left to evaluate $T_E(N,X;\alpha)$.

For $Nf\ne\square$  in \eqref{T_E}, we apply Lemma \ref{PVlemma}, and the contribution of these terms is $O(q^{-g+X/2}g^2)$. 
So, by the Perron formula,
\begin{align}\label{RV=0}
T_E(N,X;\alpha)& =\sum_{\substack{f \in \mathcal{M}_{\leq X} \\ Nf= \square}} \frac{ \lambda(f) }{|f|^{1/2+\alpha}}+O(q^{-g+X/2}g^2)\nonumber\\
&=\frac{1 }{2\pi i}\oint_{|u|=r} \sum_{\substack{Nf = \square}} \frac{ \lambda(f)u^{\deg(f)}}{|f|^{1/2+\alpha}}\frac{du}{u^{X+1}(1-u)} +O(q^{-g+X/2}g^2),
\end{align}
for $r<1$. We can write the sum in the integrand in terms of its Euler product as
\begin{align}\label{def_bm}
&\prod_{Q} \bigg( \sum_{j+\text{ord}_Q(N)\ \text{even}}\frac{ \lambda(Q^j) u^{j \deg(Q)}}{|Q|^{(1/2+\alpha)j}} \bigg)=\mathcal{A}_E(N;u)\,\mathcal{L} \Big(\text{Sym}^2 E, \frac{u^2}{q^{1+2\alpha}}\Big),
\end{align} where $\mathcal{A}_E(N;u)$ is some Euler product which is uniformly bounded for $|u|\leq q^{1/2-\varepsilon}$. We shift the contour in \eqref{RV=0} to $|u|=q^{1/2-\varepsilon}$, encountering a simple pole at $u=1$. 
Thus
\begin{equation}
T_E(N,X;\alpha)=\mathcal{A}_E(N;1)\,L(\text{Sym}^2 E, 1+2\alpha)+O(q^{-g+X/2}g^2)+O_\varepsilon ( q^{-X/2 + \varepsilon X} ).
\label{main_th1}
\end{equation}

From \eqref{toevaluate}, \eqref{main_th1} and Cauchy's residue theorem we obtain
\begin{align*}
\frac{1}{|\mathcal{P}_{2g+1}|} \sum_{P\in \mathcal{P}_{2g+1}} \epsilon^{-} L'(E \otimes \chi_P,\tfrac{1}{2}) &=2(\log q)\big(\mathcal{A}_E(1;1)-\epsilon_{2g+1} \epsilon(E)\mathcal{A}_E(M;1)\big)L(\text{Sym}^2 E, 1)g\\
&\quad+O_\varepsilon\big(g^{1/4+\varepsilon}(2g-X)^2\big)+O(q^{-g+X/2}g^2)+O_\varepsilon ( q^{-X/2 + \varepsilon X} ).
\end{align*}
Chossing $X=2g-[100 \log g]$ the theorem follows.
 
 \section{Checking the coefficients}
Here we check that the asymptotic formula we obtain in Theorem \ref{mainthm} agrees with the conjecture in \cite{A2}. From  \cite{A2} we have that the term involving $g^2$ in the asymptotic formula should be equal to
\begin{align}
 \frac{3 (2g)^2}{24}  & A(\tfrac12;0,0)+ \frac{6 (2g)^2}{24} A(\tfrac12;0,0)+ \frac{3 (2g)^2}{24 \log q} \big(  A_1(\tfrac12;0,0) + A_2(\tfrac12;0,0) \big), \label{coeff_check}
 \end{align}
 where
 \begin{align} \label{a_formula}
 A & (\tfrac12;z_1,z_2) = \prod_Q \Big(1- \frac{1}{|Q|^{1+z_1+z_2}} \Big)\Big(1- \frac{1}{|Q|^{1+2z_1}} \Big) \Big(1- \frac{1}{|Q|^{1+2z_2}} \Big)  \\
 & \times \frac{1}{2} \bigg(\Big(1-\frac{1}{|Q|^{1/2+z_1}} \Big)^{-1} \Big(1-\frac{1}{|Q|^{1/2+z_2}} \Big)^{-1} +\Big(1+\frac{1}{|Q|^{1/2+z_1}} \Big)^{-1} \Big(1+\frac{1}{|Q|^{1/2+z_2}} \Big)^{-1} \bigg),\nonumber
 \end{align}
 and
 $$ A_1(\tfrac12;0,0)=\frac{ \partial}{\partial z_1} A(\tfrac12;z_1,z_2)\Big|_{z_1=z_2=0}, $$
 and a similar expression holds for $A_2(1/2;0,0)$. Since $A_1(1/2;0,0)=A_2(1/2;0,0)$, it follows that the coefficient of $g^2$ in \eqref{coeff_check} is equal to
 \begin{equation} \frac{3A(1/2;0,0)}{2}+\frac{A_1(1/2;0,0)}{\log q}. \label{coeff_3}
 \end{equation}
 Using equation \eqref{a_formula}, we have
 $$A(\tfrac12;0,0)= \frac{1}{\zeta_q(2)},$$ and
 $$\frac{A_1(1/2;0,0)}{(\log q) A(1/2;0,0)} =\sum_Q \frac{3 \deg(Q)}{|Q|-1} - \sum_Q \frac{  (3|Q|+1) \deg(Q)}{|Q|^2-1}= \sum_Q \frac{2 \deg(Q)}{|Q|^2-1} = \frac{2}{q-1},$$ where the last identity follows from the logarithmic expression of $\zeta_q(s)$. 
 Then equation \eqref{coeff_3} simplifies to 
 $$\frac{3}{2 \zeta_q(2)}+\frac{2}{\zeta_q(2)(q-1)}= \frac{3}{2}+\frac{1}{2q},$$
 which matches the coefficient in Theorem \ref{mainthm}.
 \bibliographystyle{amsalpha}

\bibliography{prime_bibl} 
 
\kommentar{ }
\end{document}